\documentclass[reqno,12pt]{amsart}
\usepackage{amsmath,amsthm,enumerate,graphics,pstricks,amsfonts,latexsym,amsopn,verbatim,amscd,amssymb}
\usepackage{psfrag}
\usepackage[all]{xy}

\theoremstyle{plain}
\newtheorem{thm}{Theorem}[section]

\newtheorem{lem}[thm]{Lemma}

\newtheorem{prop}[thm]{Proposition}

\theoremstyle{definition}

\newtheorem{rem}[thm]{Remark}

\newtheorem{concl}[thm]{Conclusion}

\begin{document} 

\title[On Leinster groups of order $pqrs$]{On Leinster groups of order $pqrs$} 

\author[S. J. Baishya ]{Sekhar Jyoti Baishya} 
\address{S. J. Baishya, Department of Mathematics, Pandit Deendayal Upadhyaya Adarsha Mahavidyalaya, Behali, Biswanath-784184, Assam, India.}

\email{sekharnehu@yahoo.com}




\begin{abstract}
A finite group   is said to be a Leinster group if the sum of the orders of its normal subgroups equals twice the order of the group itself. Let $p<q<r<s$ be primes. We prove that if $G$ is a Leinster group of order $p^2qr$, then $G \cong Q_{20}\times C_{19}$ or $Q_{28} \times C_{13}$. We also prove that no group of order $pqrs$ is Leinster.
\end{abstract}


\maketitle

\section{Introduction and Preliminary results} \label{S:intro}
A number is perfect if the sum of its divisors equals twice the number itself. In 2001, T. Leinster \cite{leinster}, developed and studied a group theoretic analogue of perfect numbers. A finite group  is said to be a perfect group (not to be confused with the one which is equal to its commutator subgroup) or an immaculate group or a Leinster group if the sum of the orders of its normal subgroups equals twice the order of the group itself. The Leinster groups have a beautiful connection with the perfect numbers. Obviously, in the case of cyclic groups and less obviously in the case of dihedral groups. Clearly, a finite cyclic group $C_n$ is Leinster if and only if its order $n$ is a perfect number. In fact, the nilpotent Leinster groups are precisely the finite cyclic groups whose orders are perfect numbers. On the other hand, the Leinster dihedral groups are in one to one correspondence with odd perfect numbers. It may be mentioned here that, till now it is not known whether there are infinitely many Leinster groups or not. Another interesting fact is that upto now, only one odd order Leinster group is known, namely $(C_{127} \rtimes C_7) \times C_{3^4.{11}^2.{19}^2.113}$. It was discovered by F. Brunault, just one day after the question on existence of odd order Leinster groups was asked by Tom Leinster in Mathoverflow \cite{brunault}.  More information on this and the related concepts can be found in the works of S. J. Baishya \cite{baishya1}, S. J. Baishya and A. K. Das \cite{baishya2}, A. K. Das \cite{das}, M. T$\check{\rm a}$rn$\check{\rm a}$uceanu \cite{t1, t2},  T. D. Medts and A. Mar$\acute{\rm o}$ti \cite{maroti}, etc.

Given a finite group $G$, $\sigma(G), \tau(G), G'$ and $Z(G)$ denotes   the sum of the orders of the normal subgroups, the number of normal subgroups , the commutator subgroup, and the center of $G$ respectively. For any prime $l$, we have used the symbol $T_l$ to denote a $l$-Sylow subgroup of a group $G$.

The following theorems will be  used repeatedly to obtain our results:

\begin{thm}{(Theorem 4.1 \cite{leinster})} \label{aqt}
If $G$ is a group with $\sigma(G) \leq 2\mid G \mid$, then any abelian quotient of $G$ is cyclic. 
\end{thm}

\begin{thm}{(Corollary 4.2 \cite{leinster})} \label{alg}
The abelian Leinster  groups are precisely the cyclic groups $C_n$ of order $n$ with $n$ perfect. 
\end{thm}

\begin{thm}{(Proposition 3.1, Corollary 3.2 \cite{leinster})} \label{mul}
If $G_1$ and $G_2$ be two groups with $(\mid G_1 \mid , \mid G_2 \mid)=1$, then $\tau(G_1 \times G_2)=\tau(G_1)\tau(G_2)$ and $\sigma(G_1 \times G_2)=\sigma(G_1)\sigma(G_2)$. 
\end{thm}

\begin{thm} {(Proposition 3.1, Theorems 3.4, 3.5, 3.7 \cite{baishya1})} \label{sjb}
If $G (\neq C_7 \rtimes C_8)$ is a Leinster group of order $pqrs$, $p, q, r, s$ being primes (not necessarily distinct), then $\tau(G)>7$. 
\end{thm}

\begin{thm} {(Lemma 4, p. 303 \cite{zumud}}) \label{zumud1}
If a finite group $G$ has an abelian normal subgroup of prime index $p$, then $\mid G \mid=p \mid G' \mid \mid Z(G) \mid$. 
\end{thm}

\begin{thm} {(Theorem15 \cite{sqf}}) \label{sqfree}
Let $G$ be a finite group, and let $p$ be the smallest prime divisor of $\mid G \mid$. Let $Q$ be a $p$-Sylow subgroup of $G$. If $Q$ is cyclic, then $Q$ has a normal complement in $G$.
\end{thm}

\section{ Leinster groups of order $p^2qr$}

It is easy to verify that  $Q_{20}\times C_{19}$ and $Q_{28} \times C_{13}$ are Leinster groups. In this connection, a natural question arises: Is there any other Leinster group whose order is of the form $p^2qr$, $p<q<r$ being primes? We obtain the answer of this question and found that $Q_{20}\times C_{19}$ and $Q_{28} \times C_{13}$ are the only Leinster groups whose orders are of the form $p^2qr$.

We begin with the following elementary remark:

\begin{rem}\label{rem125}
If $G$ is a Leinster group, then the number of odd order normal subgroups of $G$ must be even. Hence, for any odd order Leinster group $G$,  $\tau(G)$ is even.
\end{rem}  

The following Lemmas will be used to get the main result of this section.

\begin{lem}\label{506}
If  $G$ is a Leinster group of  order $p^2qr$, where $p<q<r$ are primes, then $T_p \ntriangleleft G$. 
\end{lem}

\begin{proof}
In view of Theorem \ref{alg}, $G$ is non-abelian. Now, if $T_p \lhd G$, then by Correspondence theorem, $G$ has an abelian normal centralizer say $N$ of index $q$ and consequently, by Theorem \ref {zumud1}, we have $p^2qr=q \mid G' \mid \mid Z(G) \mid$. 

Now, if $q>3$, then $G$ has  an abelian centralizer say $K$ of index $r$. Clearly, $N \cap K =Z(G)$ and $G=NK$. Therefore $\mid Z(G) \mid =p^2$ and hence  $\mid G' \mid =r$. But then, $G$ has a normal subgroup of order $qr$, which implies $G \cong T_p \times (C_r \rtimes C_q)$. In the present scenario, if $T_p=C_{p^2}$, then by Theorem \ref{mul},  $\tau(G)=9$ and consequently, by Remark \ref{rem125}, $\mid G \mid =4qr$. Now, using Theorem \ref{mul} again, we have $8qr=\sigma(G)=\sigma(T_2) \times \sigma (C_r \rtimes C_q)=7(1+r+qr)$, which is impossible. Again, if $T_p= C_p \times C_p$, then $\sigma (G)> 2 \mid G \mid$. 

Next, suppose $q=3$. Then we have $\mid G' \mid \mid Z(G) \mid=4r$. Now, if $\mid G' \mid=4$, then $G$ has a normal subgroup of order $12$ and consequently, $G \cong A_4 \times C_r, Q_{12} \times C_r$ or $D_{12} \times C_r$, which are not Leinster. Again, if $\mid G' \mid=4r$, then $\sigma(G)  < 2 \mid G \mid$. Finally, if $\mid G' \mid=2, r$ or $2r$, then by Correspondence theorem, $G$ has a normal subgroup of index $2$, which is a contradiction. Hence $T_p \ntriangleleft G$. 
\end{proof}

\begin{lem}\label{6}
If $G$ is a Leinster group of  order $p^2qr$, $p<q<r$ being  primes,  then  $\mid G' \mid \neq r$.
\end{lem}

\begin{proof}
If $\mid G' \mid = r$, then in view of Theorem \ref{aqt}, Lemma \ref{506} and Correspondence theorem, $G$ has unique normal subgroup for each of the orders $r, pr, qr, p^2r, pqr$ and $G$ cannot have any normal subgroup of order $p^2, p^2q$. Note that $G$ can have at the most one normal subgroup of order $pq$. Suppose $G$ has a normal subgroup $N$ of order $pq$. Let $K$ be the normal subgroup of order $pr$. Now, if $\mid N \cap K \mid=1$, then exactly one of $N$ or $K$ is cyclic, noting that  $\mid G' \mid = r$. Consequently, $\tau(G)=10$. Again,  if $\mid N \cap K \mid=p$, then also we have $\tau(G)=10$. Therefore from the definition of Leinster groups, we have $p^2qr=1+p+q+pq+r+pr+qr+p^2r+pqr$. But then $(p-1)pqr=(1+p)(1+q)+(1+p+q+p^2)r$, which is a contradiction. Therefore $G$ cannot have a normal subgroup of order $pq$. Consequently, in view of Theorem \ref{sjb}, using the definition of Leinster groups, we have $p^2qr=1+p+r+pr+qr+p^2r+pqr$ or $1+q+r+pr+qr+p^2r+pqr$, which is again a contradiction. Hence  $\mid G' \mid \neq r$.
\end{proof}

\begin{lem}\label{60}
If $G$ is  a Leinster group of  order $p^2qr$, $p<q<r$ being  primes, then  $\mid G' \mid \neq qr$.
\end{lem}

\begin{proof}
If $\mid G' \mid = qr$, then in view of Theorem \ref{aqt}, Lemma \ref{506} and Correspondence theorem, $G$ has unique normal subgroup for each of the orders $qr, pqr$ and  $G$ cannot have normal subgroup of order $p^2, p^2q, p^2r$. Note that $G$ can have at the most one normal subgroup for each of the orders $pq, pr$. If $G$ has no normal subgroup of order $pq$ or $pr$, then $\tau(G) \leq 7$, which is a contradiction to Theorem \ref{sjb}. Therefore $G$ must have a normal subgroup $N$ of order $pq$ and a normal subgroup $K$ of order $pr$. Now, if $\mid N \cap K  \mid=1$, then $G=(C_q \rtimes C_p) \times (C_r \rtimes C_p)$. Therefore using the definition of Leinster groups, we have  $p^2qr=1+q+r+pq+pr+qr+pqr$, which implies $r=\frac{1+(p+1)q}{q(p^2-p-1)-(p+1)}$. In the present situation, one can easily verify that $p>7$, which is a contradiction. Therefore $\mid N \cap K  \mid=p$ and consequently, from the definition of Leinster groups, we have 
$p^2qr=1+p+q+r+pq+qr+pr+pqr=(1+p)(1+q)(1+r)$, which is again impossible.
Hence  $\mid G' \mid \neq qr$.
\end{proof}

\begin{lem}\label{61}
If $G$ is  a Leinster group of  order $p^2qr$, $p<q<r$ being  primes, then  $\mid G' \mid \neq pq$.
\end{lem}

\begin{proof}
If $\mid G' \mid = pq$, then in view of Theorem \ref{aqt} and Correspondence theorem, $G$ has unique normal subgroups, say, $H$ and $K$ of order $p^2q$ and $pqr$ respectively. Moreover, $G$ cannot have normal subgroups of order $qr$ and $p^2r$. Note that $G$ has unique normal subgroup of order $pq$ since $K$ is unique and can have at the most one normal subgroup of order $pr$. Now, suppose $G$ has a normal subgroup $N$ of order $pr$. Then $ \mid G' \cap N \mid=p$, otherwise  $G=G' \times N$, which is a contradiction since $\mid G' \mid=pq$.  Therefore $G$ has a normal subgroup of order $p$. Consequently, $G'$ is cyclic and hence $G$ has a cyclic normal subgroup of index $p$. It now follows from Theorem \ref{zumud1}, that $ \mid Z(G) \mid=r$ and consequently, $T_q \ntriangleleft G$. In the present scenario, in view of Lemma \ref{506}, using the definition of Leinster groups, we have  $p^2qr=1+p+r+pq+pr+p^2q+pqr$ and hence $r=\frac{p^2q+pq+p+1}{p^2-pq-p-1}$. In this case, one can verify that $p>7$, which is impossible. Therefore $G$ cannot have a normal subgroup of order $pr$. But then $\tau (G)\leq 7$, which is again impossible by Theorem \ref{sjb}. Hence $\mid G' \mid \neq pq$.
\end{proof}

\begin{lem}\label{62}
If $G$ is a Leinster group of  order $p^2qr$, $p<q<r$ being  primes, then  $\mid G' \mid \neq pr$.
\end{lem}

\begin{proof}
If $\mid G' \mid = pr$, then in view of Theorem \ref{aqt} and Correspondence theorem, $G$ has unique normal subgroup for each of the orders $p^2r$ and $pqr$. Moreover, $G$ cannot have normal subgroups of order $qr$ and $p^2q$. Let $K$ be the normal subgroup of order $pqr$. It is easy to see that $G$ can have at the most one normal subgroup of order $pq$ and $pr$, noting that $K$ is unique. 

Now, if $N$ is a normal subgroup of $G$ of order $pq$, then   $\mid G' \cap N \mid=p$, otherwise  $G=G' \times N$, which is a contradiction since $\mid G' \mid=pr$. It now follows that $G$ has a normal subgroup of order $p$ and hence $N$ is cyclic.  Now, if $T_r \lhd G$, then $K$ is a cyclic normal subgroup of $G$ of index $p$. Therefore by Theorem \ref{zumud1}, we have  $\mid Z(G) \mid=q$, which is impossible. Consequently, in view of Theorem \ref{sjb} and  Lemma \ref{506}, using the definition of Leinster groups, we have  $p^2qr=1+p+q+pq+pr+p^2r+pqr=(1+p)(1+q)+r(p+p^2+pq)$, which is impossible. Therefore $G$ cannot have a normal subgroup of order $pq$. But then, again we have  $\tau (G)\leq 7$, which is a contradiction to Theorem \ref{sjb}. Hence $\mid G' \mid \neq pr$.
\end{proof}

Now, we are ready to state the following theorem:

\begin{thm}\label{56}
If  $G$ is a Leinster group of  order $p^2qr$, where $p<q<r$ are primes, then $G \cong Q_{20}\times C_{19}$ or $Q_{28} \times C_{13}$.
\end{thm}

\begin{proof}
In view of Theorem \ref{alg}, $G$ is non-abelian. Clearly, $G$ cannot be simple and consequently, $G$ is solvable, which implies $G' \neq G$. 

Now, if $\mid G' \mid=pqr$ or $p^2r$, then $\tau(G) \leq 7$, which is impossible by Theorem \ref{sjb}. 

Next, suppose $\mid G' \mid = p^2q$. In this situation, if $G$ has more than one normal subgroup of order $pq$, then $T_q \lhd G$ and hence $T_r \ntriangleleft G$. Now, from the definition of Leinster group, we have $\mid G \mid=12r$ and consequently, $\mid G \mid=60$ or $132$, which is impossible by GAP \cite{gap}. Therefore $G$ can have at the most one normal subgroup of order $pq$. But then $\tau(G) \leq 7$, which is again impossible by Theorem \ref{sjb}. 

Again, if $\mid G' \mid=p$, then by Theorem \ref{aqt} and Correspondence theorem, $T_p \vartriangleleft G$, which is a contradiction to Lemma \ref {506}. Therefore, in view of Lemma \ref{506}, Lemma \ref{6}, Lemma \ref{60}, Lemma \ref{61}, Lemma \ref{62}, we have $\mid G' \mid=q$. In the present scenario, by Theorem \ref{aqt} and Correspondence theorem, $G$ has unique normal subgroup for each of the following orders $q, pq, qr, p^2q, pqr$. Also, note that $G$ cannot have normal subgroup of order $p^2r$. 

Let $N$ be the normal subgroup of $G$ of order $p^2q$. Now, if $T_r \ntriangleleft G$, then $G$ cannot have a normal subgroup of order $pr$. In this situation, from the definition of Leinster groups, we have  $p^2qr=1+p+q+pq+qr+p^2q+pqr=(1+p)(1+q)+q(r+p^2+pr)$, noting that by Theorem \ref{sjb}, we have $\tau(G)>7$ , which is impossible.  Therefore $T_r \lhd G$. In the present scenario, if $G$ don't have a normal subgroup of order $p$, then from the definition of Leinster groups, we have $p^2qr=1+q+r+pr+pq+qr+p^2q+pqr$, i.e., $r=\frac{p^2q+pq+q+1}{p^2q-pq-p-q-1}$ or $p^2qr=1+q+r+pq+qr+p^2q+pqr$, i.e., $r=\frac{p^2q+pq+q+1}{p^2q-pq-q-1}$.  But in both the cases, one can verify that $p>7$, which is impossible.  Therefore $G$ must have a normal subgroup of order $p$ and consequently,  using the definition of Leinster groups, we have  $p^2qr=1+p+q+r+pq+qr+pr+p^2q+pqr$. In the present situation, one can verify that $p=2$ and consequently, $qr=3+7q+3r$, which implies  $r=(3+7q)/(q-3)=7+24/(q-3)$. Hence $q=5, r=19$ or $q=7, r=13$. Now, using GAP \cite{gap}, we have  $G \cong Q_{20}\times C_{19}$ or $Q_{28} \times C_{13}$.
\end{proof}

\section{ Leinster groups of order $pqrs$}

Given any primes $p<q<r<s$, by  \cite[ Corollary 3.2, Theorem 3.6]{baishya1}, the only Leinster group of order $pq$ is $C_6$ and  the only Leinster group of order $pqr$ is $S_3 \times C_5$. In this section, we consider the groups of order $pqrs$ and prove that no group of order $pqrs$ is Leinster. The following lemmas will be used to establish our result.

\begin{lem}\label{51}
Let $H$ be a group of squarefree order and $p$ be any odd prime, $p \nmid \mid H \mid $. If $G=H \times C_{2p}$, then $G$ is not Leinster.
\end{lem}

\begin{proof}
In view of Theorem \ref{alg}, $G$ is non-abelian. Now, if  $ \mid H \mid$ is even, then $H$ has a subgroup $N$ of index $2$. In the present scenario, $N \times C_{2p}$ and $H \times C_p$ are two distinct subgroups of $G$ of index $2$ and hence $G$ is not Leinster. Next, suppose $ \mid H \mid$ is odd. Then using Theorem \ref{mul}, we have $3 \mid\sigma (G)$ and so if $G$ is Leinster, then $3 \mid \mid G \mid$. Therefore $G$ will have a normal subgroup of index $2$ and a normal subgroup of index $3$. Hence $G$ is not Leinster.
\end{proof}

\begin{lem}\label{52}
If $G$ is a Leinster group of  order $pqrs$, $p<q<r<s$ being primes, then  $8 \leq \tau (G) \leq 10 $. 
\end{lem}

\begin{proof}
In view of Theorem \ref{alg}, $G$ is non abelian. Consequently, $\tau(G) \leq 12$, noting that every normal subgroup of $G$ is uniquely determined by its order.

Now, suppose $\tau (G)=11$. In view of Theorem \ref{mul}, $G$ can have at the most $3$ normal subgroups of prime index and at the most $3$ normal subgroups of order product of two primes. Consequently, $G$ is not leinster.

Next, suppose $\tau (G)=12$. Then we have $G= N_{ab} \times C_{cd}$, where  $N_{ab}$ is the unique normal subgroup of order $ab$ and $a, b, c, d \in \lbrace p, q, r, s \rbrace$. Now, if $\mid G \mid$ is odd, then using Theorem \ref{mul}, we have $4 \mid \sigma (G)$ and hence $G$ is not Leinster. Next, suppose $\mid G \mid$ is even. In the present situation, if $\mid C_{cd} \mid $ is odd, then $\mid N_{ab} \mid$ must be even and by Theorem \ref{mul}, we have $8 \mid \sigma (G)$, and hence $G$ is not Leinster. On the other hand if $\mid C_{cd} \mid $ is even, then by Lemma \ref{51}, $G$ is not Leinster. Now, the result follows from Theorem \ref{sjb}.
\end{proof}

We begin with the case when $\tau(G)=8$.

\begin{rem}\label{rem133}
Let $G$ be a  group of  order $pqrs$, $p<q<r<s$ being primes and $\tau (G) = 8$. If $G$ is Leinster, then we have:
\begin{enumerate}
	\item  $\mid G' \mid = qr, qs$ or $rs$.
	\item  $p=2$.
	\item  $G$ has atleast one normal subgroup of prime order and has atleast $2$ normal subgroups whose orders are product of two primes.
	\item  $3 \nmid \mid G \mid$.
\end{enumerate} 
\end{rem}

\begin{proof}
a) If $\mid G' \mid$ is prime, then $\tau(G) \geq 9$. Next, suppose $G'$ is of prime index. Then in view of Theorem \ref{sqfree}, we have $\mid G' \mid=qrs$. Now, if $p=2$,  then by Theorem \ref{zumud1}, $\mid Z(G) \mid=1$, and hence  $G$ is a dihedral group, which is impossible by \cite[Example 2.4]{leinster}. Hence $\mid G \mid$ is odd. In the present scenario, from the definition of Leinster groups, we have $\mid G \mid =1+n_1+n_2+n_3+n_4+n_5+n_6$, where $n_1>n_2>n_3>n_4>n_5>n_6$ are the orders of the normal subgroups of $G$.
Consequently,  we must have $n_2 > \frac{\mid G \mid}{9}$, which is impossible. Therefore $ \mid G' \mid=qr, qs$ or $rs$, noting that we cannot have $G =G'$. 

b)  If $p\neq 2$, then by (a), we have  $\sigma(G)-\mid G \mid \leq \frac{\mid G \mid}{3}+  \frac{\mid G \mid}{5}+ \frac{\mid G \mid}{15}+ \frac{4\mid G \mid}{21} < \mid G \mid$.

c) In view of (a), $G$ has exactly $2$ normal subgroups of prime index. Now, the result follows from the fact that $\tau(G)=8$.

d) If $\mid G \mid = 2.3rs$, then in view of (a), we have $\mid G' \mid = 3r$ or $3s$.

Now, suppose $\mid G' \mid = 3s$. Then in view of (c), by the definition of Leinster groups, we have 
$\mid G \mid =1+3rs+2.3s+3s+x+y+z$, where $x, y, z$ are the orders of the remaining normal subgroups such that $x$ is a product of two primes. In the present situation, if $x=2s$, then by Remark \ref{rem125}, we have $\lbrace y, z \rbrace = \lbrace 2, s \rbrace$, which is impossible.
Similarly, if $x=2.3$, then by Remark \ref{rem125}, we have $\lbrace y, z \rbrace = \lbrace 2, 3 \rbrace$, which is impossible.
Finally, if  $x=rs$, then by Remark \ref{rem125}, we have $\lbrace y, z \rbrace = \lbrace 3, s \rbrace$ or $\lbrace r, s \rbrace$, which is also impossible.
Therefore, we must have $x=3r$, noting that in the present scenario, we cannot have a normal subgroup of order $2r$.
But then, by Remark \ref{rem125}, we have $\lbrace y, z \rbrace = \lbrace 3, r \rbrace$ or $\lbrace 3, s \rbrace$, which is again impossible.

Next, suppose $\mid G' \mid = 3r$.  Then in view of (c), by the definition of Leinster groups, we have 
$\mid G \mid =1+3rs+2.3r+3r+x+y+z$, where $x, y, z$ are the orders of the remaining normal subgroups such that $x$ is a product of two primes. In the present situation, if $x=2r$, then by Remark \ref{rem125}, we have $\lbrace y, z \rbrace = \lbrace 2, r \rbrace$, which is impossible.
Similarly, if $x=2.3$, then by Remark \ref{rem125}, we have $\lbrace y, z \rbrace = \lbrace 2, 3 \rbrace$, which is impossible.
Again, if  $x=rs$, then by Remark \ref{rem125}, we have $\lbrace y, z \rbrace = \lbrace 3, r \rbrace$ or $\lbrace 3, s \rbrace$, which is also impossible.
Finally, if $x=3s$, then by Remark \ref{rem125}, we have $\lbrace y, z \rbrace = \lbrace 3, r \rbrace$ or $\lbrace 3, s \rbrace$, which is again impossible. Hence  $3 \nmid \mid G \mid$, noting that in the present scenario, we cannot have a normal subgroup of order $2s$.
\end{proof}

\begin{lem}\label{511}
If $G$ is a  group of  order $pqrs$, $p< q<r<s$ being primes, and $\tau (G) = 8$, then $G$ is not Leinster.
\end{lem}

\begin{proof}
If $qr > 2s$, then in view of Remark \ref{rem133}, we have
\begin{align*}
\sigma(G)- \mid G \mid \leq 1+ \frac{\mid G \mid}{2}+\frac{\mid G \mid}{5}+\frac{\mid G \mid}{10}+\frac{\mid G \mid}{14}+\frac{2\mid G \mid}{22} < \mid G \mid.
\end{align*}

Again, if $qr < 2s$, then in view of Remark \ref{rem133}, we have
\begin{align*}
\sigma(G)- \mid G \mid \leq 1+ \frac{\mid G \mid}{2}+\frac{\mid G \mid}{5}+\frac{\mid G \mid}{10}+\frac{\mid G \mid}{14}+\frac{2\mid G \mid}{35} < \mid G \mid.
\end{align*}
\end{proof}

Now, we consider the case where $\tau (G) = 9$.

\begin{rem}\label{rem123}
Let $G$ be a group of  order $pqrs$, $p<q<r<s$ being primes and  $\tau (G) = 9$. If $G$ is Leinster, then we have:
\begin{enumerate}
	\item  $p=2$.
	\item  $T_2 \ntriangleleft G$.
    \item  $\mid G' \mid = qr, qs$ or $rs$.
	\item  $G$ has exactly $2$ normal subgroups of prime index and has exactly $2$ normal subgroups of prime order.
	\item $3 \nmid \mid G \mid$.
\end{enumerate} 
\end{rem}  

\begin{proof}
a) It follows from Remark \ref{rem125}.

b) It follows from Theorem \ref{mul}, noting that by Theorem \ref{sqfree}, $G$ has a normal subgroup of index $2$.

c) Since $\tau(G)=9$, therefore $\mid G' \mid$ cannot be prime. Again, if $G'$ is of prime index, then in view of Theorem \ref{sqfree}, we have $\mid G' \mid =qrs$. In the present scenario, by Theorem \ref{zumud1}, $\mid Z(G) \mid=1$, and hence  $G$ is a dihedral group, which is  impossible by \cite[Example 2.4]{leinster}.  Therefore we must have $\mid G' \mid=qr, qs$ or $rs$, noting that we cannot have $G =G'$. 

d)  In view of (c), $G$ has exactly two normal subgroups of prime index. In the present scenario, if $G$ has less than two normal subgroups of prime order, then using Theorem \ref{mul},  we have $\tau(G)>9$. Hence, $G$ has exactly $2$ normal subgroups of prime order, noting that if $G$ has more than $2$ normal subgroup of prime order, then also  we have $\tau(G)>9$.

e) If $\mid G \mid=2.3rs$, then in view of (c), we have $\mid G' \mid=3r$ or $3s$. 
Now, suppose $\mid G' \mid=3r$.  In the present situation, if $G$ has a normal subgroup of order $rs$, then in view of (d) and Remark \ref{rem125}, from the definition of Leinster groups, we have 
\begin{align*}
6rs=1+3rs+2.3r+3r+rs+r+2.3+3,
\end{align*}
 which is impossible. 

Next, suppose $G$ has a normal subgroup of order $3s$. Then  in view of (d) and Remark \ref{rem125}, from the definition of Leinster groups, we have
\begin{align*}
6rs=1+3rs+2.3r+3r+3s+3+x+y,
\end{align*}
where $x, y$ are the orders of the remaining normal subgroups such that $x$ is a product of two primes. But then, clearly, $x$ cannot be odd, otherwise $y$ will also be odd. Again, $x \neq 2r, 2s$, noting that  $\mid G' \mid=3r$. Therefore $x=2.3$, which is again impossible. 

Therefore, we must have $\mid G' \mid=3s$. In the present situation, if $G$ has a normal subgroup of order $rs$, then in view of (d) and Remark \ref{rem125}, from the definition of Leinster groups, we have 
\begin{align*}
6rs=1+3rs+2.3s+3s+rs+s+2.3+3,
\end{align*}
which is impossible. 

Next, suppose $G$ has a normal subgroup of order $3r$. Then  in view of (d) and Remark \ref{rem125}, from the definition of Leinster groups, we have
\begin{align*}
6rs=1+3rs+2.3s+3s+3r+3+x+y,
\end{align*}
where $x, y$ are the orders of the remaining normal subgroups such that $x$ is a product of two primes. But then, clearly, $x$ cannot be odd, otherwise $y$ will also be odd. Again, $x \neq 2r, 2s$, noting that  $\mid G' \mid=3s$. Therefore $x=2.3$, which is again impossible.  Therefore  we have $3 \nmid \mid G \mid$.
\end{proof}

\begin{lem}\label{257}
If $G$ is a  group of  order $pqrs$, $p< q<r<s$ being primes, and $\tau (G) = 9$, then $G$ is not Leinster.
\end{lem}

\begin{proof}
If $qr > 2s$, then in view of  Remark \ref{rem123}, we have
\begin{align*}
\sigma(G)- \mid G \mid \leq \frac{\mid G \mid}{2}+\frac{\mid G \mid}{5}+\frac{\mid G \mid}{10}+\frac{\mid G \mid}{14}+\frac{\mid G \mid}{22}+\frac{3\mid G \mid}{70} < \mid G \mid.
\end{align*}

Again, if $qr < 2s$, then in view of  (e) and (f) of Remark \ref{rem123}, we have
\begin{align*}
\sigma(G)- \mid G \mid \leq \frac{\mid G \mid}{2}+\frac{\mid G \mid}{5}+\frac{\mid G \mid}{10}+\frac{\mid G \mid}{14}+\frac{\mid G \mid}{35}+\frac{3\mid G \mid}{70} < \mid G \mid.\end{align*}
\end{proof}

Finally, we consider the case where $\tau(G)=10$.

\begin{rem}\label{rem923}
Let $G$ be a group of  order $pqrs$, $p<q<r<s$ being primes and  $\tau (G) = 10$. If $G$ is Leinster, then we have:
\begin{enumerate}
	\item  $p=2$.
	\item $3 \nmid \mid G \mid$.
	\item  $T_2 \ntriangleleft G$.
	\item  $\mid G' \mid = qr, qs$ or $rs$.
\end{enumerate} 
\end{rem}

\begin{proof}
a)  Suppose, $p>2$. It is easy to see that $G'$ cannot be of prime index,  noting that in view of Theorem \ref{mul}, $G$ can have at the most three normal subgroup whose order is product of two primes. 

Now, suppose $\mid G' \mid$ is prime. Note that, in view of Theorem \ref{sqfree}, $\mid G' \mid \neq p$. Next, suppose $\mid G' \mid = q$. In this situation, if $T_p \lhd G$, then $q \mid p+1$, which is a contradiction. Again, if $T_r \lhd G$ or $T_s \lhd G$, then by Correspondence theorem, we have a normal subgroup of order $rs$, which is impossible. Therefore  $\mid G' \mid \neq q$. On the other hand, if  $\mid G' \mid = s$, then we have $s \leq {r+1} $, which is impossible. Therefore $\mid G' \mid = r$ and consequently, in view of Theorem \ref{mul}, from the definition of Leinster groups, we have
\begin{align*}
 \sigma(G)-\mid G \mid=1+r+s+pr+qr+rs+pqr+prs+qrs.
\end{align*}
Now, suppose $pqr > rs$. Then $pr<qr<rs<pqr<prs<qrs$. In the present situation, clearly we have $qrs \geq \frac{\mid G \mid}{3}, prs \geq \frac{\mid G \mid}{5}, pqr \geq \frac{\mid G \mid}{11}, rs \geq \frac{\mid G \mid}{15}, qr \geq \frac{\mid G \mid}{33}$ and $pr \geq \frac{\mid G \mid}{55}$. But then, $\sigma(G)-\mid G \mid < \mid G \mid$.

Next, suppose $pqr< rs$. Then $pr<qr<pqr<rs<prs<qrs$. In the present situation, clearly we have $qrs \geq \frac{\mid G \mid}{3}, prs \geq \frac{\mid G \mid}{5}, rs \geq \frac{\mid G \mid}{15},, pqr \geq \frac{\mid G \mid}{17},  qr \geq \frac{\mid G \mid}{51}$ and $pr \geq \frac{\mid G \mid}{85}$. But then also, we have $\sigma(G)-\mid G \mid < \mid G \mid$.

Therefore  $\mid G' \mid$ has to be product of two primes. In the present scenario, $G$ can have exactly two normal subgroups of prime index and in view of Theorem \ref{mul}, $G$ can have exactly three normal subgroups of order product of two primes. 

Now, suppose $qr >ps$. Then we have 
\begin{align*}
\sigma(G)-\mid G \mid \leq 1+q+r+s+qr+qs+rs+prs+qrs.
\end{align*}

 In the present situation, clearly we have $qrs \geq \frac{\mid G \mid}{3}, prs \geq \frac{\mid G \mid}{5}, rs \geq \frac{\mid G \mid}{15}, qs \geq \frac{\mid G \mid}{21}, qr \geq \frac{\mid G \mid}{33}$ and $s \geq \frac{\mid G \mid}{105}$. But then,  $\sigma(G)-\mid G \mid < \mid G \mid$.

Therefore we must have $qr< ps$, and consequently, 
\begin{align*}
\sigma(G)-\mid G \mid \leq 1+q+r+s+ps+qs+rs+prs+qrs.
\end{align*}

 In the present situation, clearly we have $qrs \geq \frac{\mid G \mid}{3}, prs \geq \frac{\mid G \mid}{5}, rs \geq \frac{\mid G \mid}{15}, qs \geq \frac{\mid G \mid}{21}, ps \geq \frac{\mid G \mid}{35}$ and $s \geq \frac{\mid G \mid}{105}$. But then also, we have $\sigma(G)-\mid G \mid < \mid G \mid$. Therefore $p=2$.

b) In view of (a), suppose $ \mid G \mid=2.3rs$. Note that, by Correspondence theorem, $T_2 \ntriangleleft G$. In the present scenario, in view of Theorem \ref{sqfree}, one can verify that $\mid G' \mid=3r$ or $3s$,  and consequently, $T_3, T_r, T_s \lhd G$, noting that $\mid G' \mid \neq 3, r, s$.  

Now, suppose $\mid G' \mid=3r$. Then we  have $G=T_s \times (C_{3.r} \rtimes C_2)$. But then, from the definition of Leinster groups, using Theorem \ref{mul}, we have $12rs=(s+1)(1+r+3+3r+6r)$, which is a contradiction. 

Next, suppose  $\mid G' \mid=3s$. Then we  have $G=T_r \times (C_{3s} \rtimes C_2)$. But then, again from the definition of Leinster groups, using Theorem \ref{mul}, we have  $12rs=(r+1)(1+s+3+3s+6s)$, which is again a contradiction. Therefore $3 \nmid \mid G \mid$.

c) It follows from (b), using Theorem \ref{mul}.

d) Clearly, $G'$ cannot be of prime index. Now, suppose $\mid G' \mid$ be a prime. In the present situation, one can verify that $\mid G' \mid =q$ or $r$ and consequently, $s=\frac{1+r+3q+3qr}{qr-3q}$ or $s=\frac{1+3r+3qr}{qr-3r-1}$. But in both the cases, one can verify that $q>13$, which is impossible. Hence the result follows, noting that we cannot have $G=G'$.
\end{proof}

\begin{lem}\label{757}
Let $G$ be a group of  order $pqrs$, $p<q<r<s$ being  primes. If $\tau (G)=10$, then $G$ is not leinster. 
\end{lem}

\begin{proof}
If $qr > 2s$, then in view of  Remark \ref{rem923}, we have
\begin{align*}
\sigma(G)- \mid G \mid \leq \frac{\mid G \mid}{2}+\frac{\mid G \mid}{5}+\frac{\mid G \mid}{10}+\frac{\mid G \mid}{14}+\frac{\mid G \mid}{22}+\frac{4\mid G \mid}{70} < \mid G \mid.
\end{align*}

Again, if $qr < 2s$, then in view of  (e) and (f) of Remark \ref{rem123}, we have
\begin{align*}
\sigma(G)- \mid G \mid \leq \frac{\mid G \mid}{2}+\frac{\mid G \mid}{5}+\frac{\mid G \mid}{10}+\frac{\mid G \mid}{14}+\frac{\mid G \mid}{35}+\frac{4\mid G \mid}{70} < \mid G \mid.\end{align*}
\end{proof}

Combining Lemma \ref{52}, Lemma \ref{511}, Lemma \ref{257}, Lemma \ref{757}, we now have the following theorem:

\begin{thm}\label{711}
If $G$ is a  group of  order $pqrs$, $p<q<r<s$, being  primes, then $G$ is not Leinster. 
\end{thm}

We have the following result on Leinster group of order $p^3q$ for any primes $p, q$, which also follows from \cite[Proposition 3.8]
{maroti}.

\begin{prop}\label{911}
If $G$ is a Leinster group of order $p^3q$, $p, q$ being primes, then  $G \cong C_7 \rtimes C_8$. 
\end{prop}

\begin{proof}
In view of Theorem \ref{alg}, $G$ is non-abelian. In the present scenario,  one can easily verify that $p<q$ and $T_q \lhd G$. Now, by \cite[Corollary 3.3]{maroti}, we have $T_p$ is cyclic and consequently,  from \cite[Theorem 3.4, Theorem 3.5]{baishya1}, we have $\tau(G)=7$. But then, using remark \ref{rem125}, we have $p=2$ forcing $q=7$. Now, the result follows using GAP \cite{gap}.
\end{proof}

\begin{concl}\label{1711}
The objective of this research was to investigate the structure of Leinster groups of order $pqrs$ for any primes $p, q, r, s$ not necessarily distinct. It is well known that no group of order $p^2q^2$ is Leinster  \cite [Proposition 2.4]{baishya1}. On the other hand, in view of  Proposition \ref{911}, we have $C_7 \rtimes C_8$ is the only Leinster group of order $p^3q$.  In this paper,  for any primes $p<q<r<s$, we have studied the cases where order of $G$ is $p^2qr$ or $pqrs$. However, we were unable to give any information for Leinster groups of order $pq^2r$ and $pqr^2$. As such we leave it as an open question to our readers.
\end{concl}

\end{document}